\documentclass[10pt]{article}
\usepackage{amssymb,amsmath,graphicx,verbatim,eufrak,amsthm,hyperref}
\newtheorem{thm}{Theorem}

\newtheorem{lem}[thm]{Lemma}
\newtheorem{prop}[thm]{Proposition}

\theoremstyle{definition}
\newtheorem{dfn}[thm]{Definition}
\theoremstyle{remark}

\numberwithin{equation}{section}

\newcommand{\ignore}[1]{}

\newcommand{\abs}[1]{\left\vert#1\right\vert}

\newcommand{\set}[1]{\left\{#1\right\}}
\newcommand{\Set}[2]{ \left\{#1 \ \big| \ #2 \right\} }
\newcommand{\br}[1]{\left[#1\right]}

\newcommand{\sr}[1]{\left(#1\right)}

\newcommand{\R}{\mathbb{R}}

\newcommand{\eps}{\varepsilon}
\newcommand{\To}{\longrightarrow}

\DeclareMathOperator*{\E}{\mathbb{E}}

\renewcommand{\Pr}{}
\let\Pr\relax
\DeclareMathOperator*{\Pr}{\mathbb{P}}

\newcommand{\1}[1]{\mathbf{1}_{\set{ #1 } }}

\newcommand{\eqdef}{\stackrel{\mathrm{def}}{=}}

\def\squareforqed{\hbox{\rlap{$\sqcap$}$\sqcup$}}
\def\qed{\ifmmode\squareforqed\else{\unskip\nobreak\hfil
\penalty50\hskip1em\null\nobreak\hfil\squareforqed
\parfillskip=0pt\finalhyphendemerits=0\endgraf}\fi}

\newcommand{\G}{\Gamma}

\newcommand{\M}{\mathcal{M}}

\newcommand{\supp}{\mathrm{supp}}

\newcommand{\Ss}{\mathbb{S}}
\begin{document}

\title{When Do Random Subsets Decompose a Finite Group? }

\author{Ariel Yadin
\thanks{Department of Mathematics, The
Weizmann Institute of Science,  POB 26, Rehovot 76100, ISRAEL;
\texttt{ariel.yadin@weizmann.ac.il}} }

\date{ }

\maketitle

\begin{abstract}
Let $A,B$ be two random subsets of a finite group $G$.  We consider the event that the
products of elements from $A$ and $B$ span the whole group; i.e. $\set{ A B \cup B A = G }$.
The study of this event gives rise to a group invariant we call $\Theta(G)$.  $\Theta(G)$ is between
$1/2$ and $1$,
and is $1$ if and only if the group is abelian.  We show that a phase transition occurs as the size
of $A$ and $B$ passes $\sqrt{ \Theta(G) |G| \log |G|}$; i.e. for any $\eps>0$,
if the size of $A$ and $B$ is less than
$(1-\eps)\sqrt{ \Theta(G) |G| \log |G|}$, then with high probability $A B \cup B A \neq G$.  If
$A$ and $B$ are larger than $(1+\eps)\sqrt{ \Theta(G) |G| \log |G|}$ then $A B \cup B A = G$
with high probability.
\end{abstract}

\section{Introduction}

Let $G$ be a finite group.  Two subsets $A,B \subset G$ are said to be
a \emph{decomposition} for $G$ if $A B \cup B A = G$, where
$$ A B \eqdef \Set{ a b }{ a \in A, b \in B } . $$

In \cite{Gady}, Kozma and Lev proved that for any finite group $G$, there always
exists a decomposition $A,B$ for $G$, such that $\abs{A} \leq c \sqrt{|G|}$, $\abs{B} \leq c \sqrt{|G|}$
(where $c>0$ is some explicit constant, see \cite{Gady} for details).  We consider a similar
question, but where the sets $A,B$ are randomly chosen.

Let $G$ be a finite group of size $n \geq 3$.  Let $a_1, a_2, \ldots,
a_k, b_1, b_2, \ldots, b_k$ be $2k$ random elements (perhaps with repetitions) chosen
independently from $G$, and let $A = \set{ a_1 , \ldots, a_k}$ and
$B = \set{ b_1, \ldots, b_k}$.
We investigate the event that $A$ and $B$ are a decomposition for $G$. Denote the probability of
this event by
$$ P(G,k) = \Pr \br{ A B \cup B A = G } . $$
Since $P(G,k)$ is monotone in $k$, it seems natural to ask whether a phase transition
occurs, and if so, then what is the critical value.  It turns out that
there exists a group invariant $\Theta(G) \in (1/2, 1]$ (defined in Section
\ref{scn:  group invariant} below), such that the critical value exists
and is equal to $\sqrt{ \Theta(G) n \log n }$,
as stated in our main result:

\begin{thm} \label{thm:  main thm}
Let $G_n$ be a family of groups such that
$$ \lim_{n \to \infty} |G_n| = \infty . $$
For all $n$, let $C_n = \sqrt{ \Theta(G_n) |G_n| \log |G_n| }$.
Then for any $\eps > 0$,
$$ \lim_{n \to \infty} P(G_n, \lceil (1+\eps) C_n \rceil ) = 1 , $$
and
$$ \lim_{n \to \infty} P(G_n, \lfloor (1-\eps) C_n \rfloor ) = 0 . $$
\end{thm}

The proof of Theorem \ref{thm:  main thm} follows from Lemmas \ref{lem:  first moment} and
\ref{lem:  second moment}. Actually, it can be seen from these lemmas that the window
of the transition is smaller than stated by Theorem \ref{thm:  main thm}.
Before we move to the proofs of these lemmas, we define the group invariant $\Theta(G)$, and
elaborate on some of its properties.

\section{A Group Invariant} \label{scn:  group invariant}

For group theory background see \cite{Rotman}.

Say we are interested in measuring how close a group is to being abelian.  It seems reasonable
to try and associate a number, say $\rho(G)$, to each group $G$, such that $\rho(G)$ has the following properties:
\begin{itemize}
\item $\rho(G) \in [0,1]$.
\item $\rho(G) = \rho(G')$, if $G$ and $G'$ are isomorphic as groups.
\item $\rho(G) =1$ if and only if $G$ is abelian.
\end{itemize}

Perhaps the first ``probabilistic'' quantity that comes to mind is the probability that two randomly chosen
elements commute.  If $a,b$ are two random independent uniformly chosen elements from a finite group $G$,
then
\begin{equation} \label{eqn:  random elmts commute}
\Pr \br{ ab = ba } = \sum_{x \in G} \frac{|C(x)|}{|G|^2} ,
\end{equation}
where $C(x) = \set{ g \in G \ : \ gx=xg }$ denotes the centralizer of $x$ in $G$.
If we view $G$ as acting on itself by conjugation, then $C(x)$ is the set of all elements
that fix $x$.  Also, the number of different orbits is just the number of conjugacy classes of $G$.
Thus, by Burnside's counting lemma (see \cite{Rotman} Chapter 3, page 58), $\Pr \br{ ab=ba } =
R(G)/|G|$, where $R(G)$ is the number of conjugacy classes in $G$.
(An alternative proof can be given through character theory,
using the Schur orthogonality relations.)

In this note, we define a different group invariant, $\Theta(G)$.
As it turns out, $\Theta(G) \in (1/2,1]$, and
$\Theta(G) = 1$ if and only if $G$ is abelian.  $\Theta(G)$ arises naturally when considering the question
that two random sets form a decomposition of a group $G$, as seen in Theorem \ref{thm:  main thm}.

We use the notation $C(x) = \Set{ g \in G}{ gx = xg }$ to denote the centralizer
of $x \in G$.  Note that $C(x)$ is a subgroup of $G$.

Let $G$ be a group of order $n$.
Since for any $x \in G$, $2 \leq \abs{C(x)} \leq n$, the function $f:[1/2,1] \to \R$
$$ f(\xi) = 2 \xi \log n - \log \sum_{x \in G} \exp \sr{
\xi \log n \cdot \frac{\abs{C(x)}}{n} } $$
is negative at $1/2$, non-negative at $1$, and continuous monotone increasing
on $[1/2,1]$.
Indeed,
$$ f(1/2) = \log n - \log  \sum_{x \in G} \exp \sr{
\frac{\log n}{2} \cdot \frac{\abs{C(x)}}{n} } \leq \log n - \log \sr{ n \cdot e^{\log n /n} } = - \frac{\log n}{n} . $$
$$ f(1) = 2 \log n - \log  \sum_{x \in G} \exp \sr{
\log n \cdot \frac{\abs{C(x)}}{n} } \geq 2 \log n - \log \sr{ n \cdot e^{\log n} } = 0 . $$
$$ f'(\xi) = 2 \log n - \frac{ \sum_{x \in G} \exp \sr{
\xi \log n \cdot \frac{\abs{C(x)}}{n} } \cdot \frac{\log n \abs{C(x)}}{n} }{
\sum_{x \in G} \exp \sr{ \xi \log n \cdot \frac{\abs{C(x)}}{n} } } \geq
2 \log n - \log n  > 0 . $$
Thus, the following is well defined:
\begin{dfn} \label{dfn:  Theta}
Let $G$ be a finite group of order $n$.
Define $\Theta = \Theta(G)$ to be the unique number in $[1/2,1]$ satisfying:
\begin{equation} \label{eqn:  dfn of Theta}
2 \Theta \log n = \log \sum_{x \in G} \exp \sr{ \Theta \log n \cdot \frac{\abs{C(x)}}{n} } .
\end{equation}
\end{dfn}

\paragraph{Remark.} $\Theta(G)$ is the solution of equation (\ref{eqn:  dfn of Theta})
If $x_1,x_2,\ldots,x_R$ are representatives of the conjugacy classes of $G$,
the sum in the logarithm of the right hand side of (\ref{eqn:  dfn of Theta}) can be written as
$$ \sum_{i=1}^R |[x_i]| \exp \sr{ \xi \log n \cdot \frac{1}{|[x_i]|} } , $$
where $[x_i]$ is the conjugacy class of $x_i$. This sum may remind
some readers of the ``zeta function'' studied by Liebeck and
Shalev, see e.g. \cite{LiebeckShalev, LiebeckShalev2}.  Their zeta
function is also used in the context of probabilistic group
theory.  We use the main result from \cite{LiebeckShalev2}
regarding this ``zeta function'' in Proposition \ref{prop: Theta}
below.

The following proposition provides some properties of $\Theta(G)$.  The proposition roughly shows that
$\Theta(G)$ measures, in some sense, how ``abelian'' a group is.  The properties of $\Theta$
are not essential to the proof of Theorem \ref{thm:  main thm}, and so some readers may
wish to skip to Section \ref{scn:  Suen's inequality}.

\begin{prop} \label{prop:  Theta}
Let $G$ be a group of order $n$.
\begin{enumerate}
    \item \label{enum:  i}
    Let $Z(G)$ be the center of $G$; i.e. $Z(G) = \Set{ g \in G}{ \forall \ x \in G \ : \
    gx = xg}$.  Then,
    $$ \Theta(G) \geq \frac{ \log \abs{Z(G)} }{ \log n} , $$
    and
    \begin{equation} \label{eqn:  upper bound Theta}
    \Theta(G) \leq \max \set{ \frac{2}{3} \sr{1 + \frac{\log 2}{\log n} }  ,
    \frac{\log \abs{Z(G)} + \log 2 }{\log n} }  . \nonumber
    \end{equation}
    \item $G$ is abelian if and only if $\Theta(G) = 1$ (so the lower bound in \ref{enum:  i} is tight).
    \item Let $R = R(G)$ be the number of conjugacy classes of $G$ (this is also the number of
    irreducible representations of $G$). Then,
    $$ \Theta(G) \geq \frac{1}{2- R/n} > 1/2 . $$
    \item Let $G = D_{2m}$, the dihedral group of order $n=2m$.  Then,
    $$ \frac{2}{3} \cdot \sr{1 - \frac{\log 2}{\log n} } \leq \Theta(D_{2m}) \leq
    \frac{2}{3} \cdot \sr{1 + \frac{\log 2}{\log n} } . $$
    (This implies that the upper bound in \ref{enum:  i} is tight.)
    \item Let $G = S_m$, the group of all permutations on $m$ letters.  So $n=m!$.  Then,
    $$ \Theta(S_m) = \frac{1}{2} + o(1) . $$
    \item \label{enum:  limit}
    Let $1/2 \leq \alpha < 1$.  Then, there exists a sequence of groups $\set{G_n}$,
    such that
    $$ \lim_{n \to \infty} \Theta(G_n) = \alpha . $$
    \item Let $G$ be a simple non-abelian group. Then,
    $$ \Theta(G) = \frac{1}{2} + o(1) . $$
\end{enumerate}
\end{prop}

\begin{proof}
Let $\Theta = \Theta(G)$.
\begin{enumerate}
\item For any $x \in Z(G)$, we have that $\abs{C(x)} = n$.  Thus,
$$ 2 \Theta \log n \geq \log (\abs{Z(G)} e^{ \Theta \log n }) = \log
\abs{Z(G)} + \Theta \log n . $$
This proves the lower bound.

Note that since $C(x)$ is a subgroup, $\abs{C(x)}$ must divide $|G|$.  Thus, for any
$x \not\in Z(G)$, we have that $\abs{C(x)} \leq n/2$.  Thus,
$$ n^{2 \Theta} \leq \abs{Z(G)} \cdot n^{\Theta} + (n-\abs{Z(G)}) \cdot n^{\Theta/2} \leq
n^{\Theta/2} \cdot 2 \max \set{ \abs{Z(G)} n^{\Theta/2} , n } . $$
This proves the upper bound.

\item Assume towards a contradiction
that $\Theta = 1$ and that $G$ is not abelian.  Then,
there exists $x \in G$ such that $\abs{C(x)} < n$.
Since $\abs{C(x)}$ must divide $n$, we get that
$\abs{C(x)} \leq n/2$.  Thus, by the definition of $\Theta$,
$$ n^2 \leq (n-1) n + n^{1/2} = n^2 - n + n^{1/2} , $$
a contradiction.

The other direction follows by \ref{enum:  i}, since if $G$ is abelian, $\abs{Z(G)} = n$.

\item By Burnside's Lemma, or by Schur's orthogonality relations, one can show that
$$ \sum_{x \in G} \abs{C(x)} = n \cdot R . $$
Using Jensen's inequality on the convex function $\exp \sr{ \frac{\Theta \log n}{n} \cdot \xi }$,
$$ 2 \Theta \log n = \log \sum_{x \in G} \exp \sr{ \frac{\Theta \log n}{n} \cdot \abs{C(x)} }
\geq \log n \exp \sr{ \frac{\Theta \log n}{n} \cdot R } = \log n +
\Theta \log n \cdot \frac{R}{n} . $$
The assertion follows.

\item \label{enum:  D_2m}
The dihedral group of order $n=2m$ is
$$ D_{2m} = \left< x,y \ : \ x^m = y^2 = 1 \ , \ yxy=x^{-1} \right> =
\set{ x^i , y x^i \ : \ i=0,1,\ldots,m-1 } . $$
One can check that the following holds:
$$ \begin{array}{lr}
i \not\in \set{ 0, m/2} & C(x^i) = \set{ 1,x,\ldots,x^{m-1} } , \\
i \in \set{0,m/2} & C(1) = C(x^{m/2}) = D_{2m} , \\
\textrm{ if } m \textrm{ is even } & C(y x^i) = \set{1, x^{m/2} , y x^i, y x^{i+m/2} } , \\
\textrm{ if } m \textrm{ is odd } & C(y x^i) = \set{1, y x^i} .
\end{array} $$
Thus, $Z(D_{2m}) = \set{1}$ if $m$ is odd, and $Z(D_{2m}) = \set{1, x^{m/2}}$, if $m$
is even.  So we get the upper bound by \ref{enum:  i}.

On the other hand, considering the elements $1,x,\ldots,x^{m-1}$, we have that
$$ n^{2 \Theta} \geq m \cdot n^{\Theta/2} = \frac{1}{2} n^{1+\Theta/2} , $$
which implies the lower bound.


%

\item We use the following notation:  If $c = (i_1,i_2, \ldots,i_s) \in S_m$ is a cycle, and
$\tau \in S_m$ is any permutation, then denote $c^{\tau} = (\tau(i_1), \tau(i_2), \ldots, \tau(i_s))$
(note that $c^{\tau} = \tau c \tau^{-1}$).  For a permutation $\sigma \in S_m$ denote by $\supp(\sigma) =
\set{ j \in [m] \ : \ \sigma(j) \neq j }$ the support of $\sigma$.  $\abs{\sigma} = \abs{\supp(\sigma)}$ denotes
the size of the support.

Let $\sigma \in S_m$, and write $\sigma = c_1 c_2 \cdots c_\ell$, where $c_i$ are cycles,
ordered by their size from largest to smallest (i.e. $\abs{c_i} \geq \abs{c_j}$ for all $i \leq j$).
Let $s = \abs{c_1}$ be the size of the largest cycle in the decomposition.
Let $r \geq 1$ be the index such that $\abs{c_i} = s$ for all $1 \leq i \leq r$, and $\abs{c_i} < s$
for all $i > r$.

Set
$$ S = \bigcup_{i=1}^r \supp(c_i) . $$

If $\tau \in C(\sigma)$, then $\tau \sigma \tau^{-1} = \sigma$.  But it can easily be seen that
$$ \tau \sigma \tau^{-1} = c_1^{\tau} c_2^{\tau} \cdots c_\ell^{\tau} . $$
Since $\abs{c_i} = \abs{c_i^{\tau}}$, we get that
for any $j \in S$ we must have that $\tau(j) \in S$.  Thus,
$$ \abs{ C(\sigma) } \leq \abs{ \Set{ \tau \in S_m }{ \tau(S) = S } } = \abs{S}! (m-\abs{S})! . $$
If $r < \ell$, then since $\abs{c_\ell} \geq 2$, we have that
$\abs{S} \leq m-2$.  Thus,
\begin{equation} \label{eqn:  2(m-2)!}
\abs{C(\sigma)} \leq m! \cdot \frac{2}{m(m-1)} < m! \cdot \sr{ \frac{e}{m} }^2  .
\end{equation}

Assume that $r=\ell$.  Then either $\sigma$ is a cycle of length $m$, or a cycle of length $m-1$,
or $\sigma$ is the product of cycles of equal length.

If $\sigma$ is the identity, then $\abs{C(\sigma)} = m!$.
If $\sigma$ is a cycle of length $m$ then $\abs{C(\sigma)} = m$.  If $\sigma$ is a
cycle of length $m-1$ then $\abs{C(\sigma)} = m-1$.

So we are left with the case where $\sigma = c_1 c_2 \cdots c_r$ and $\abs{c_i} = m/r$ for all
$1 \leq i \leq r$.  Note that in this case,
$$ C(\sigma) \subseteq \Set{ c_1' c_2' \cdots c_r' }{ \textrm{ all } c_i' \textrm{ are cycles of length } m/r } . $$
Thus,
\begin{equation} \label{eqn:  (e/m)^2}
\abs{C(\sigma)} = \frac{m!}{\sr{\frac{m}{r}}^r r!} < m! \cdot \sr{ \frac{e}{m} }^2 .
\end{equation}

Combining (\ref{eqn:  2(m-2)!}) and (\ref{eqn:  (e/m)^2}) we get (for $n=m!$),
\begin{eqnarray*}
n^{2 \Theta} & \leq & n^\Theta + (m-1)! n^{\Theta/(m-1)!} + m(m-2)! n^{\Theta/m(m-2)!}
\\
& & \ \ + \sr{m! - (m-1)! - m(m-2)!-1} \cdot n^{\Theta e^2/m^2 } \\
& \leq & (n-1) \cdot (1+o(1)) + n^{\Theta}
\end{eqnarray*}
which shows that $\Theta \leq \frac{1}{2} + o(1)$ (as $m$ tends to infinity).

\item Let $\alpha \in [1/2 , 1)$.  For all integers $m$, let $n_m = m!$ and
$k_m = \lfloor n_m^{\alpha/(1-\alpha)} \rfloor$.
So,
$$ n_m^{\alpha/(1-\alpha)} \cdot \sr{1- 1/n_m} \leq k_m \leq (k_m n_m)^{\alpha} . $$
Let $G_m = C_{k_m} \times S_m$, where $C_{k_m}$ is the cyclic group of order
$k_m$.  Note that for $c \in C_{k_m}$ and $\sigma \in S_m$ the centralizer
of $(c,\sigma)$ in $G_m$ is the set $C_{k_m} \times C(\sigma)$.  Thus, using
the calculations for $S_m$ in the previous proof, for $\Theta = \Theta(G_m)$,
\begin{eqnarray*}
|G_m|^{2 \Theta} & \leq & k_m \cdot \big( |G_m|^{\Theta} +
(m-1)! |G_m|^{\Theta/(m-1)!} + m(m-2)! |G_m|^{\Theta/m(m-2)!} \\
& & \ \
+ (n_m - 1 - (m-1)! - m(m-2)!) |G_m|^{\Theta e^2/m^2} \big) \\
& \leq & k_m |G_m|^{\Theta} + (1+o(1)) k_m n_m .
\end{eqnarray*}
Since $\alpha + \Theta > 1$, we get that
$ |G_m|^{2 \Theta} \leq (2+o(1)) (k_m n_m)^{\Theta + \alpha} ,$
which implies that $\Theta(G_m) \leq \alpha + o(1)$.

On the other hand $|G_m|^{2 \Theta} \geq k_m |G_m|^{\Theta} \geq (1-1/n_m) \cdot n_m^{\alpha/(1-\alpha)} |G_m|^{\Theta}$.
Hence 
$\Theta(G_m) \geq \alpha - o(1)$.

\item Let $G$ be a finite simple non-abelian group.
Let $\M$ be the set of all maximal subgroups of $G$.
Consider the following ``zeta function'' (defined in \cite{KanLub}, and studied
further in \cite{LiebeckShalev, LiebeckShalev2}):
$$ \zeta_G(s) = \sum_{M \in \M} [G:M]^{-s} . $$
Theorem 1.1 of \cite{LiebeckShalev2} states that for any $s>1$,
$$ \zeta_G(s) \To 0 \quad \textrm{ as } \quad |G| \to \infty . $$
Since $G$ is simple non-abelian, if $x$ is not the identity in $G$, then $C(x)$ is a proper subgroup.
Since any proper subgroup of $G$ is contained in a maximal subgroup, we get that
$$ \frac{|C(x)|}{|G|} = \sqrt{ [G:C(x)]^{-2} } \leq \sqrt{ \zeta_G(2) } . $$
So for all $x \neq 1$ in $G$ we get that $\frac{|C(x)|}{|G|} = o(1)$.  Plugging this into the
definition of $\Theta = \Theta(G)$ we get that
$$ n^{2 \Theta} \leq n^{\Theta} + (n-1) \cdot n^{\Theta o(1)} \leq n^{\Theta} + n^{1+o(1)} , $$
which implies that $\Theta \leq \frac{1}{2} + o(1)$. \qedhere
\end{enumerate}
\end{proof}

In the proof of Proposition \ref{prop:  Theta}, \ref{enum:  limit}, we use the product of a cyclic group
with the symmetric group to obtain different values of $\Theta$.  This idea raises the following

\paragraph{Open Problem.} Show that for any \emph{abelian} group $H$ and any group $G$,
$\Theta(H \times G) \geq \Theta(G)$.

\section{Suen's inequality} \label{scn:  Suen's inequality}

One of the main tools we use to prove our results is a correlation inequality by Suen
(see Theorem \ref{thm:  Suen} below).

\paragraph{Graph Notation.}
For a graph $\G = (V,E)$ write $v \sim u$ if $\set{v,u}$ is an
edge.  For subsets $S,T \subseteq V$, write $S \sim T$ if there
exists an edge between $S$ and $T$.  Thus, $S \not\sim T$ means
that there is no edge between $S$ and $T$.  $v \sim S$ means there
is an edge between $v$ and some element of $S$.

\begin{dfn} \label{dfn:  dependency graph}
Let $\set{X_i}_{i=1}^N$ be a collection of random variables.
A graph $\G = (V,E)$ is a \emph{dependency graph} of $\set{X_i}_{i=1}^N$ if:
$\G = (V,E)$ is an undirected graph on the vertex set $V = \set{1,\ldots,N}$
such that for any two disjoint subsets $S,T \subset V$, if $S
\not\sim T$ then the two families $\set{X_i}_{i \in S}$ and
$\set{X_i}_{i \in T}$ are independent of each other.  (In some
texts $\G$ is called a superdependency digraph.)
\end{dfn}

The following is a result of Suen, slightly improved by Janson
(see \cite{Janson(Suen), Suen}).

\begin{thm}[Suen's inequality] \label{thm:  Suen}
Let $X_1, \ldots, X_N$ be $N$ Bernoulli random variables, and let
$S_N = \sum_{i=1}^N X_i$.
Let $\G$ be a \emph{dependency graph} of $\set{X_i}_{i=1}^N$.

Define
$$ \Delta = \Delta \sr{\Gamma, \set{X_i}_{i=1}^N } = \frac{1}{2} \sum_{i=1}^N \sum_{j \sim i} \E \br{X_i
X_j} \prod_{k \sim \set{i,j}} \sr{ 1- \E \br{X_k} }^{-1} , $$
and
$$ \Delta^* = \Delta^* \sr{\Gamma, \set{X_i}_{i=1}^N } = \frac{1}{2} \sum_{i=1}^N \sum_{j \sim i} \E \br{X_i}
\E \br{X_j} \prod_{k \sim \set{i,j}} \sr{ 1- \E \br{X_k} }^{-1} .
$$ %

Then,
$$ \Pr \br{ S_N = 0 } \leq e^{\Delta} \prod_{i=1}^N \sr{ 1- \E \br{ X_i} } , $$
$$ \Pr \br{ S_N = 0 } \geq \sr{ 1 - \Delta^* e^{\Delta} } \prod_{i=1}^N \sr{ 1- \E \br{ X_i}
} . $$ %
\end{thm}

\section{Preliminaries}

Let $G$ be a finite group of size $n \geq 3$.  Let $a_1, a_2, \ldots,
a_k, b_1, b_2, \ldots, b_k$ be $2k$ random elements chosen
independently from $G$, and let $A = \set{ a_1 , \ldots, a_k}$ and
$B = \set{ b_1, \ldots, b_k}$.

We use the notation $[k] = \set{1,2,\ldots,k}$.

Let $V = [k] \times [k]$.  For $(i,j) \in V$, define
$$ I_{(i,j)} (x) = \1{ x = a_i \cdot b_j \ \textrm{ or } \ x = b_j \cdot a_i } . $$

\begin{dfn} \label{dfn:  the graph Gamma}
Define a graph $\G = (V,E)$ on the vertex set $V$, by the edge relation
$$ (i,j) \sim (\ell,m) \quad \Longleftrightarrow \quad
(i=\ell \textrm{ and  } j \neq m) \textrm{ or } (i \neq \ell \textrm{ and } j=m). $$
\end{dfn}

\begin{prop} \label{prop:  Gamma is dep. graph}
For any $x \in G$,
$\G$ is a dependency graph for $\set{ I_v(x)}_{v \in V}$.
\end{prop}

\begin{proof}
Let $S,T$ be disjoint subsets of $V$ such that $S \not\sim T$.
Note that the values of $\set{I_v(x)}_{v \in S}$ are completely determined by
$\Set{ a_i , b_j }{ (i,j) \in S}$, and the values of
$\set{I_v(x)}_{v \in T}$ are completely determined by
$\Set{ a_i,b_j }{ (i,j) \in T}$.  Since $S \not\sim T$ and $S \cap T = \emptyset$, by definition,
for any $(i,j) \in S$ and $(\ell,m) \in T$, we have that $i \neq j$ and $j \neq m$.  Thus,
$\Set{ a_i , b_j }{ (i,j) \in S}$ and $\Set{ a_\ell , b_m }{ (\ell,m) \in T }$ are independent.
So, the families
$\set{I_v(x)}_{v \in S}$ and $\set{I_v(x)}_{v \in T}$ are independent.
\end{proof}

\begin{dfn} \label{dfn:  graph Gamma'}
Let $x \neq y \in G$. Define $V(x,y) = V \times \set{x,y}$.
Let $\Gamma_{x,y} = (V(x,y), E_{x,y})$ be the graph defined by the edge relations
$$ (v,z) \sim (u,z') \quad \Longleftrightarrow \quad \set{v, u} \in E , $$
for all $v,u \in V$ and $z,z' \in \set{x,y}$.
\end{dfn}

For $(v,z) \in V(x,y)$, define $J(v,z) = I_v(z)$.
The following is very similar to Proposition \ref{prop:  Gamma is dep. graph},
so we omit the proof.

\begin{prop}
For any $x \neq y \in G$,
$\G_{x,y}$ is a dependency graph for $\set{ I_v(x), I_v(y)}_{v \in V} = \set{J(v,z) }_{(v,z) \in V(x,y)}$.
\end{prop}

The following Propositions prove to be useful in calculating the moments of $|A B \cup B A|$.

\begin{prop} \label{prop:  E[Iv]}
Let $x \in G$.
Let $v \in V$.  Then,
$$ \E \br{ I_v(x) } = \frac{2}{n}  \sr{ 1 - \frac{1}{2} \cdot \frac{\abs{C(x)}}{n} } . $$
\end{prop}

\begin{proof}
Let $v = (i,j) \in V$.  Since $a_i$ and $b_j$ are independent,
by the inclusion-exclusion principle,
\begin{eqnarray*}
\E \br{ I_v(x) } & = & \Pr \br{ a_i = x b_j^{-1} } + \Pr \br{ a_i = b_j^{-1} x } -
\Pr \br{ a_i = x b_j^{-1} = b_j^{-1} x } \\
& = & \frac{1}{n} + \frac{1}{n} - \frac{1}{n} \Pr \br{ b_j^{-1} \in C(x) } .  \qquad \qedhere
\end{eqnarray*}
\end{proof}

\begin{prop} \label{prop:  E[Iv * Iu]}
Let $x,y \in G$.
Let $v \in V$ and let $u \sim v$.  Then,
$$ \E \br{ I_v(x) I_u(y) } = \frac{4}{n^2} \sr{ 1 - \frac{\abs{C(x)} +
\abs{C(y)}}{2n} + \frac{\abs{C(x) \cap C(y)}}{4n} } . $$
\end{prop}

\begin{proof}
Assume that $v = (i,j)$ and $u = (i,\ell)$ for $\ell \neq j$. Conditioning on $a_i = g$,
\begin{eqnarray} \label{eqn:  ai = xg or ai = gx}
\E \br{ I_v(x) I_u(y) } & = & \Pr \br{ \sr{ x = a_i \cdot b_j \ \textrm{ or } \ x = b_j \cdot a_i }
\ \textrm{ and } \ \sr{ y = a_i \cdot b_\ell \ \textrm{ or } \ y = b_\ell \cdot a_i } } \nonumber \\
& = & \frac{1}{n} \sum_{g \in G}  \Pr \br{ b_j = x g^{-1} \ \textrm{ or } \ b_j = g^{-1} x }
\Pr \br{ b_\ell = y g^{-1} \ \textrm{ or } \ b_\ell = g^{-1} y } . \nonumber
\end{eqnarray}
Considering the four cases:  $g^{-1} \in C(x) \cap C(y)$, $g^{-1} \in C(x) \setminus C(y)$,
$g^{-1} \in C(y) \setminus C(x)$, $g^{-1} \not\in C(x) \cup C(y)$, we get that
\begin{eqnarray*}
\E \br{ I_v(x) I_u(y) }  & = & \frac{1}{n^3} \cdot \sr{
 \abs{C(x) \cap C(y)} + 4(n- \abs{C(x) \cup C(y)}) + 2 \abs{C(x) \setminus C(y)} + 2 \abs{C(y) \setminus C(x)} }  \\
 & = & \frac{1}{n^3} \sr{ 4n - 2(\abs{C(x)} + \abs{C(y)} ) + \abs{C(x) \cap C(y)} } \\
& = & \frac{4}{n^2} \sr{ 1 - \frac{\abs{C(x)} + \abs{C(y)}}{2n} + \frac{\abs{C(x) \cap C(y)}}{4n} } .
\end{eqnarray*}
The case $u = (\ell,j)$ for $\ell \neq i$ is very similar (condition on $b_j = g$).
\end{proof}

\begin{prop} \label{prop:  degree of v and (v,u)}
Let $v \in V$ and let $u \sim v$.  Then,
$$ \abs{ \set{ w \in V \ : \ w \sim v } } = 2(k-1) , $$
$$ \abs{ \set{ w \in V \ : \ w \sim \set{v,u} } } = 3(k-1) +1 . $$
\end{prop}

\begin{proof}
Assume that $v = (i,j)$.
The first assertion follows from
$$ \set{ w \in V \ : \ w \sim v } = \set{ (i,\ell) \ : \ \ell \neq j } \cup
\set{ (\ell,j) \ : \ \ell \neq i } , $$
since the above union is disjoint.

For the second assertion, assume that $u = (i,\ell)$ for $\ell \neq j$ (the proof
for $u = (\ell,j)$ for $\ell \neq i$ is very similar).
$$ \abs{ \set{ w \in V \ : \ w \sim \set{v,u} } } = \abs{ \set{ w \sim v } }
+ \abs{ \set{ w \sim u } } - \abs{ \set{ w \sim u \textrm{ and } w \sim v } } . $$
Since
$$ \set{ w \in V \ : \ w \sim u \textrm{ and } w \sim v } = \set{ (i,m) \ : \ m \neq j \textrm{ and }
m \neq \ell } , $$
we get that
$$ \abs{ \set{ w \in V \ : \ w \sim \set{v,u} } } = 4(k-1) - (k-2) = 3(k-1) +1 .
\qquad \qedhere $$
\end{proof}

\subsection{$\Delta$ and $\Delta^*$}

In order to apply Suen's inequality (Theorem \ref{thm:  Suen}), we need to calculate $\Delta$
and $\Delta^*$ as in Theorem \ref{thm:  Suen}, for the families of indicators $\set{I_v(x)}_{v \in V}$ and
$\set{J(v,z)}_{(v,z) \in V(x,y)}$.

\begin{lem} \label{lem:  Delta for Iv}
Let $x \neq y \in G$.
\begin{enumerate}
\item Let $\Delta_I(x) = \Delta(\G,\set{I_v(x)}_{v \in V})$ and
$\Delta_I^*(x) = \Delta^*(\G,\set{I_v(x)}_{v \in V})$ as in the statement of Theorem \ref{thm:  Suen}.
Then, $\Delta_I(x)$ and $\Delta_I^*(x)$ are both not larger than
$4 \cdot \frac{k^3}{n^2} \cdot \exp \sr{ \frac{6k}{n-2} }$.

\item Let $\Delta_J(x,y) = \Delta (\G_{x,y} , \set{J(v,z)}_{(v,z) \in V(x,y)})$
and $\Delta_J^*(x,y) = \Delta^* (\G_{x,y} , \set{J(v,z)}_{(v,z) \in V(x,y)})$.  Then,
$\Delta_J(x,y)$ and $\Delta_J^*(x,y)$ are both not larger than
$16 \cdot \frac{k^3}{n^2} \cdot \exp \sr{ \frac{12 k}{n-2} }$.
\end{enumerate}
\end{lem}

\begin{proof}
By Propositions \ref{prop:  E[Iv]} and \ref{prop:  E[Iv * Iu]},
for any $v \sim u$, the quantities $\E \br{ I_v(x) I_u(x) }$ and $\E \br{ I_v(x)} \E
\br{ I_u(x)}$ are bounded by $\frac{4}{n^2}$.  By Proposition \ref{prop:  degree of v and (v,u)},
$$ \prod_{w \sim \set{v,u}} \sr{ 1- \E \br{I_w(x)} }^{-1} \leq
\sr{ 1- \frac{2}{n} }^{-3k} \leq \exp \sr{ \frac{6k}{n-2} } , $$
where we have used the inequality $(1-\frac{1}{\xi})^{-1} \leq \exp \sr{ \frac{1}{\xi-1} }$, valid for
any $\xi > 1$.

Plugging this into the definitions of $\Delta_I(x)$ and $\Delta_I^*(x)$ proves the first assertion.

Note that
$$ \Delta_J(x,y) = \frac{1}{2} \sum_{v \in V} \sum_{u \sim v} \sum_{z,z' \in \set{x,y} }
\E \br{ I_v(z) I_u(z') }  \prod_{w \sim
\set{v,u}} \sr{1- \E \br{ I_w(x) } }^{-1} \sr{1-\E \br{I_w(y)} }^{-1} , $$
and
$$ \Delta_J^*(x,y) = \frac{1}{2} \sum_{v \in V} \sum_{u \sim v} \sum_{z,z' \in \set{x,y} }
\E \br{ I_v(z) } \E \br{ I_u(z') }  \prod_{w \sim
\set{v,u}} \sr{1- \E \br{ I_w(x) } }^{-1} \sr{1-\E \br{I_w(y)} }^{-1} . $$
So, as above, the second assertion follows from
$$ \sum_{z,z' \in \set{x,y} } \E \br{ I_v(z) I_u(z') } \leq \frac{16}{n^2} \quad
\textrm{ and } \quad
\sum_{z,z' \in \set{x,y} } \E \br{ I_v(z) } \E \br{ I_u(z') } \leq \frac{16}{n^2} .
\qquad \qedhere $$
\end{proof}


\section{Bounds on $|A B \cup B A |$}

In this section we provide bounds on the probability of the event that
$\set{A B \cup B A = G}$, i.e. that $A$ and $B$ are a decomposition of $G$.
Let $S = G \setminus A B \cup B A$.  Thus, $A B \cup B A = G$ if and only if $\abs{S} = 0$.
To bound the required probabilities, we bound the first and second moments of $|S|$.

\begin{lem} \label{lem:  first moment}
Let $0 \leq \psi < \log n$, and let
$k \geq \sqrt{ \Theta(G) n (\log n + \psi) }$.  Then,
$$ \Pr \br{ A B \cup B A \neq G } \leq (1+o(1)) \cdot e^{ - \Theta(G) \psi } . $$
\end{lem}

\begin{proof}
Since, by Markov's inequality,
$$ \Pr \br{ A B \cup B A \neq G } = \Pr \br{ \abs{S} \geq 1} \leq
\E \br{ \abs{S} } , $$
it suffices to bound $\E \br{ \abs{S} }$.

Note that the event $\Pr \br{ A B \cup B A = G }$ is monotone non-decreasing with $k$, so we
can assume that $k = \lceil \sqrt{ \Theta n (\log n + \psi) } \rceil$, where $\Theta= \Theta(G)$.

Now, $x \in S$ if and only if $\sum_{v \in V} I_v(x) = 0$.
By Lemma \ref{lem:  Delta for Iv},
$$ \Delta_I = \Delta_I(x) = O \sr{ k^3/n^2 } = o(1) . $$
Thus, using Suen's inequality (Theorem \ref{thm:  Suen}),
for any $x \in G$,
\begin{eqnarray*}
\Pr \br{ x \in S} & \leq & e^{\Delta_I} \cdot \sr{1-
\frac{2}{n}  \sr{ 1 - \frac{1}{2} \cdot \frac{\abs{C(x)}}{n} } }^{\abs{V}} \\
& \leq & (1+o(1)) \cdot \exp \sr{ - \frac{2 k^2}{n} + \frac{k^2 \abs{C(x)}}{n^2} }  .
\end{eqnarray*}
Summing over all $x \in G$, we get
\begin{eqnarray*}
\E \br{ \abs{S} } & \leq & (1+o(1)) \exp \sr{ - \frac{2 k^2}{n} } \cdot \sum_{x \in G}
\exp \sr{ \frac{k^2 \abs{C(x)} }{n^2} } \\
& \leq & (1+o(1) \cdot \exp \sr{ - 2 \Theta \psi - 2 \Theta \log n } \sum_{x \in G} \exp \sr{ \Theta \log n \frac{\abs{C(x)}}{n} }
\exp \sr{ \Theta \psi \frac{\abs{C(x)}}{n} } \\
& \leq & (1+o(1)) \cdot \exp \sr{ - \Theta \psi } .
\end{eqnarray*}
\end{proof}

\begin{lem} \label{lem:  second moment}
Let $0 \leq \psi < \log n$, and let
$k \leq \sqrt{ \Theta(G) n (\log n - \psi) }$.  Then,
$$ \Pr \br{ A B \cup B A = G } \leq e^{- \Theta(G) \psi } + o(1) . $$
\end{lem}

\begin{proof}
As in the proof of Lemma \ref{lem:  first moment}, we can assume that
$k = \lfloor \sqrt{ \Theta n (\log n - \psi) } \rfloor$, for $\Theta = \Theta(G)$.

We can bound the moments of $|S|$ using Suen's inequality, as in the proof of
Lemma \ref{lem:  first moment}.
To simplify the notation we will use $p_x = (1- \E \br{I_v(x)})^{k^2}$
(which does not depend on $v$, by Proposition \ref{prop:  E[Iv]}).
By our choice of $k$, since $\Delta_I(x) = o(1)$ and $\Delta_I^*(x) = o(1)$,
$\Pr \br{ x \in S } \geq (1-o(1)) \cdot p_x$ and $\Pr \br{ x \in S} \leq (1+o(1)) \cdot p_x$.
Thus,
$$ \E \br{ |S|} \geq (1-o(1)) \cdot \sum_{x \in G} p_x  . $$
Furthermore, note that for $x \neq y \in G$, since $\Delta_J(x,y) = o(1)$,
\begin{eqnarray*}
\Pr \br{ x,y \in S } & \leq & (1+o(1)) \cdot
\prod_{(v,z) \in V(x,y) } \sr{1- \E \br{ J(v,z) } }
=  (1+o(1)) \cdot p_x p_y .
\end{eqnarray*}
Hence,
\begin{eqnarray*}
\E \br{ |S|^2} & = & \sum_{x \neq y \in G} \Pr \br{ x,y \in S } + \sum_{x \in G} \Pr \br{ x \in S} \\
& \leq & (1+o(1)) \cdot \sum_{x \neq y \in G} p_x p_y + (1+o(1)) \cdot \sum_{x \in G} p_x .
\end{eqnarray*}

Now we use the Paley-Zygmund inequality:
\begin{eqnarray*}
\Pr \br{ A B \cup B A \neq G } & = & \Pr \br{ |S| > 0 }
 \geq  \frac{ \sr{ \E \br{ |S|} }^2 }{ \E \br{ |S|^2} } \\
 & \geq & (1-o(1)) \cdot \frac{ (\sum_x p_x )^2}{ \sum_{x \neq y} p_x p_y + \sum_x p_x }
 = (1-o(1)) \cdot \sr{1 -  \frac{\sum_x p_x - \sum_x p_x^2}{ \sum_{x \neq y} p_x p_y + \sum_x p_x } } \\
& = & (1-o(1)) \cdot \sr{1 -  \frac{\sum_x p_x (1-p_x)}{ \sum_{x} p_x (1 - p_x + \sum_{y} p_y) } }
\end{eqnarray*}

So it suffices to show that for all $x \in G$,
$$ \frac{1-p_x}{1 - p_x + \sum_{y} p_y} \leq (1+o(1)) \cdot e^{ - \Theta \psi}  . $$
But this follows immediately from
$$ \frac{1-p_x}{1 - p_x + \sum_{y} p_y} \leq \sr{ \sum_{y} p_y }^{-1} , $$
and from the fact that
\begin{eqnarray*}
\sum_{y \in G} p_y & \geq & \sum_{y \in G} \exp
\sr{ - \frac{2 k^2}{n-2} \cdot \sr{ 1- \frac{\abs{C(y)} }{2n} } } \\
& = & \sum_{y \in G} \exp \sr{ - \frac{2 k^2}{n} \cdot \sr{ 1- \frac{\abs{C(y)} }{2n} } - \frac{4 k^2}{n(n-2)} \cdot
\sr{1- \frac{\abs{C(y)} }{2n} } } \\
& \geq & (1-o(1)) \cdot \exp \sr{ 2 \Theta \psi - 2 \Theta \log n }
\cdot \sum_{y \in G} \exp \sr{ \Theta \log n \frac{\abs{C(y)} }{n} }
\exp \sr{ - \Theta \psi \frac{\abs{C(y)} }{n}  } \\
& \geq & (1-o(1)) \cdot \exp \sr{ \Theta \psi } ,
\end{eqnarray*}
(where we have used the inequality $1-\frac{1}{\xi} \geq \exp \sr{ - \frac{1}{\xi-1} }$, valid
for any $\xi >1$).
\end{proof}

\section{Concluding Remarks and Open Problems}

\begin{itemize}

\item One can ask whether a result similar to Theorem \ref{thm:  main thm} holds if we only require
that $A B = G$.  It can be shown that for any finite group $G$ of order $|G|=n$,
if $k \geq (1+\eps)\sqrt{ n \log n }$ then $\Pr \br{ A B  = G } = 1-o(1)$,
and if $k \leq (1-\eps) \sqrt{ n \log n}$,
then $\Pr \br{ A B  = G } = o(1)$. The proof of this is almost identical to the proof of Theorem
\ref{thm:  main thm}.

\item Another variant is to take $A$ and $B$ of different sizes.  That is, let $a_1,\ldots,a_k$ and $b_1,\ldots,b_m$ be
$k+m$ random elements of $G$, and let $A = \set{a_1,\ldots,a_k}$ and $B = \set{b_1,\ldots,b_m}$.  What can be said about the
probability $\Pr \br{ AB \cup BA = G }$?  It turns out that if $k$ and $m$ are both not too large, then the threshold is identical
to the case $m=k$.  That is, provided that $\max \set{k,m} = o \sr{ \frac{|G|}{\log |G|} }$, we can prove that
$\Pr \br{ A B \cup BA  = G } = 1-o(1)$ if $k \cdot m \geq (1+\eps) \Theta(G) n \log n$ and if $k \cdot m
\leq (1-\eps) \Theta(G) n \log n$, then $\Pr \br{ A B  \cup BA = G } = o(1)$.
Again, the proof is the same as that of Theorem \ref{thm:  main thm}.

\item We can also ask what is the probability of the event $A A  = G$.  In this case, our method breaks down for groups $G$
such that $\Theta(G)$ is very small.  That is, we can prove a phase transition in $k$ for the event $\set{ A A = G}$,
but only for families of groups $\set{G_n}$, such that $\Theta(G_n) \geq \frac{1}{2} + \frac{\log \log |G_n|}{\log |G_n|}$.
Note that in Section \ref{scn:  group invariant} it is shown that there are groups (e.g. the symmetric group) that
do not have this property.  The main problem in dealing with $A A$, is that one needs to control the size of the set
$\set{ a^2 \ : \ a \in A }$.  This means controlling the probability $\Pr \br{ a_i^2 = x }$ for all $x$.  Thus,
we have the following

\paragraph{Open Problem.}  Prove or provide a counter-example:

Let $G_n$ be a family of groups such that
$$ \lim_{n \to \infty} |G_n| = \infty . $$
For all $n$, let $a_1, a_2, \ldots,a_k$ be $k$ randomly chosen elements of $G_n$,
and let $A = \set{a_1, a_2, \ldots, a_k}$.
Let $P'(n,k) = \Pr \br{ A A = G_n }$.

For all $n$, let $C_n = \sqrt{ 2 \Theta(G_n) |G_n| \log |G_n| }$.
Then for any $\eps > 0$,
$$ \lim_{n \to \infty} P'(n, \lceil (1+\eps) C_n \rceil ) = 1 \ , \quad
\lim_{n \to \infty} P'(n, \lfloor (1-\eps) C_n \rfloor ) = 0 . $$

\item Another interesting problem, is to determine what happens inside the transition window:  As can be seen
by Lemmas \ref{lem:  first moment} and \ref{lem:  second moment}, if
$\psi(n)$ is any function tending to infinity with $n$, then
for $k \geq \sqrt{\Theta(G) n \log n} + \sqrt{n \psi(n)}$, with high probability $AB \cup BA = G$.  For
$k \leq \sqrt{\Theta(G) n \log n} - \sqrt{n \psi(n)}$, with high probability $AB \cup BA \neq G$.

The question is, what happens for $\sqrt{\Theta(G) n \log n} - \sqrt{n} < k < \sqrt{\Theta(G) n \log n}
+ \sqrt{n}$?  What can be said about the size of $AB \cup BA$ in this case?

\item Here are some further open questions, proposed by Itai Benjamini:

Let $G$ be a finite group.  Consider the family of subsets
$$ \Ss = \Set{ B \subset G }{ \exists \ A \subset G \ : \ A A = B } . $$

\begin{enumerate}
\item Determine the size of $\Ss$.
\item Sample $B \in \Ss$ from the uniform distribution.
\item Devise an (efficient) algorithm to decide whether a subset $A \subset G$
is in $\Ss$ or not.
\item Devise an (efficient) algorithm to decide whether $A \subset G$ is ``almost''
an element of $\Ss$; i.e. whether there exists $B \in \Ss$ such that $\abs{A \triangle B} = o(|G|)$.
\end{enumerate}

It will be interesting to solve some of these problems even with relaxed conditions, such as
assuming that $G$ is abelian or even cyclic.

\end{itemize}

\paragraph{Acknowledgements.} I wish to thank Itai Benjamini for suggesting
this problem and useful discussions. I thank Nir Avni for his comments on a preliminary
version of this paper, his ideas regarding the group invariant $\Theta$, and for bringing
the papers by Liebeck and Shalev to my attention.  I thank Gady Kozma for useful
discussions and his insights into simplifying some of the proofs.


\end{document}